\documentclass[12pt]{huber_article}
\usepackage[margin=1.5in]{geometry}

\usepackage{amssymb}
\usepackage{amsthm}
\usepackage{amsmath}
\usepackage{dsfont}
\usepackage{booktabs}
\usepackage{graphicx}
\usepackage{comment}

\newtheorem{theorem}{Theorem}{
\newtheorem{lemma}{Lemma}
\theoremstyle{definition}
\newtheorem{definition}{Definition}

\newcommand{\mean}{\mathbb{E}}
\newcommand{\prob}{\mathbb{P}}
\newcommand{\unifdist}{\textsf{Unif}}
\newcommand{\geodist}{\textsf{Geo}}
\newcommand{\ind}{{\mathds{1}}}
\newcommand{\rsoft}{\textsf{R}}
\newcommand{\told}{t_{\text{old}}}

\newcounter{line}

\begin{document}

%\begin{frontmatter}

\title{Fast Perfect Simulation of Vervaat \\
Perpetuities}

\author{Kirkwood Cloud \\ {\tt kacloud@gmail.com} \\[6pt]
Mark Huber \\ {\tt mhuber@cmc.edu}}

\maketitle

%\author{Kirkwood Cloud}
%\ead{kacloud@gmail.com}
%\address{Claremont McKenna College}
%\author{Mark Huber}
%\ead{mhuber@cmc.edu}
%\address{Claremont McKenna College}

\begin{abstract}
This work presents a faster method of simulating exactly from a
distribution known as a 
Vervaat perpetuity.  A parameter of the Vervaat perpetuity is 
$\beta \in (0,\infty)$.
An earlier method for simulating from this distributon
ran in time $O((2.23\beta)^{\beta}).$
This earlier method utilized dominated coupling from the past that bounded
a stochastic process for perpetuities from above.  By extending to
non-Markovian update functions, it is possible to create a new method that
bounds the perpetuities from both above and below.  This new approach
is shown to run in $O(\beta \ln(\beta))$ time.  
\end{abstract}


%\begin{keyword}
%exact simulation \sep dominated coupling from the past
%\MSC[2010] 68Q25 \sep 65C05
%\end{keyword}

%\end{frontmatter}

\section{Introduction}

A perpetuity is
defined as follows.

\begin{definition}
A {\em perpetuity} is a random variable of the form
\begin{equation}
\label{EQN:perpetuity}
Y = W_1 + W_1 W_2 + W_1 W_2 W_3 + \cdots,
\end{equation}
where the $\{W_i\}$ are an independent, identically distributed (iid)
sequence of random variables.
\end{definition}

Suppose that each $W_i$ has the same distribution as $W$ (write $W_i \sim W$). 
Then it 
also holds that $Y \sim W(1 + Y)$ for $Y$ and $W$ independent.  As pointed
out in~\cite{huber2010a}, this distributional identity also
characterizes perpetuities.

Throughout this work it will be assumed that 
$W \geq 0$ (with probability 1) 
and $\mean[W] < 1$.  These two assumptions give that 
$Y$ is nonnegative and finite with probability 1.  In fact, the Monotone
convergence theorem gives that 
\begin{equation}
\label{EQN:mean}
\mean[Y] = \frac{\mean[W]}{1 - \mean[W]}.
\end{equation}

\subsection{Vervaat perpetuities}
The running time of the classic {\tt Quickselect} algorithm of 
Hoare~\cite{hoare1961} for finding order statistics of an unsorted 
set of elements approaches a perpetuity.  If one pivot is chosen, then
asymptotically the running time approaches a perpetuity known as the 
Dickman distribution.  Write $U \sim \unifdist([0,1])$ to mean that $U$
has the uniform distribution over $[0,1]$.

\begin{definition}
The {\em Dickman distribution} is a perpetuity where 
$W \sim \unifdist([0,1])$.
\end{definition}

No closed form for this distribution is known.  
The Dickman distribution also arises in largest prime factors, and in longest
cycles in permutations.  See~\cite{hwangt2002} for more details.  The
Dickman distribution is a special case of the family of Vervaat perpetuities.

\begin{definition}
A {\em Vervaat perpetuity} is a perpetuity where $W_i \sim U^{1/\beta}$ for
some $\beta \in (0,\infty)$ for $U \sim \unifdist([0,1])$.
\end{definition}
% * <a2ndemail@gmail.com> 2016-09-20T04:15:09.769Z:
%
% ^.

(Note that some authors define the Dickman distributon as $1 + Y$ for 
$Y$ a Vervaat perpetutiy with $\beta = 1$.)
The first to simulate from the Dickman distribution was Fill, then
Devroye~\cite{devroye2001} followed with a substantially different method
based on envelope refinement of acceptance/rejection.
In~\cite{huber2010a} Fill and the second author improved upon Fill's 
method for the problem 
and applied it to general Vervaat perpetuities. Returning to the 
Dickman distribution, Devroye and Fawzi~\cite{devroyef2010}
improved the algorithm to the point where only 2.32 uniforms were
needed on average to generate one Dickman random variable.  Blanchet
and Sigman~\cite{blanchets2011} applied dominated coupling from
the past to more general perpetutities, but only showed that there
method had finite expected running time~\cite{blanchets2011}.

Returning to the Vervaat class, in the case of large $\beta$ the 
running time in~\cite{huber2010a} 
takes
$O((2.23 \beta)^\beta)$ steps on average.  This is 
fine for the Dickman distribution
where $\beta = 1$, but very bad for general Vervaat perpetuities.

The goal of this work is to develop a faster method for 
simulating random variates from the Vervaat perpetuity distribution,
particularly when $\beta \gg 1$.

\begin{theorem}
\label{THM:main}
The algorithm for generating Vervaat perpetuities from 
Section~\ref{SEC:method}
uses $T$ uniform random variates, where
\[
\mean[T] \leq O(\beta\ln(\beta)).
\]
\end{theorem}

This work is organized as follows.  Section~\ref{SEC:method} gives
the algorithm and proves correctness.  Section~\ref{SEC:runtime} then
proves the running time bound.  
%Section~\ref{SEC:generalize} then shows
%how this method can be employed for problems outside of the Vervaat 
%family of perpetutities.

\section{The method}
\label{SEC:method}

The method used here is a variant of coupling from the past (cftp).
In~\cite{huber2010a}, monotone dominated cftp was employed to draw samples 
from~\eqref{EQN:perpetuity}.  
We begin by 
presenting the intuition behind Propp and Wilson's 
monotone cftp~\cite{proppw1996}.  

\subsection{Monotone cftp}

Recall that a Markov chain
is a stochastic process such that the next state depends only on the current
state, and not on the past history of the process.  That means that the
next state of the process can be determined from the current state and some
independent randomness using a deterministic function called an update
function.

\begin{definition}
For a Markov chain $\{X_i\}$ with state space $\Omega$, 
$\phi:\Omega \times [0,1] \rightarrow \Omega$ is an {\em update function}
if for $U \sim \unifdist([0,1])$, $[X_{t + 1}|X_t = x_t] \sim \phi(x_t,U)$.
\end{definition}

In general update functions can employ much more general randomness than 
a single uniform, such as multiple uniforms or even an iid sequence of 
uniforms.  For notational simplicity this description 
supposes that the randomness comes from a single uniform, but everything
said here applies to the more general case.

Now suppose that $\preceq$ is a partial order on the state space $\Omega$
(so $(a \preceq a)$, $(a \preceq b)$ and $(b \preceq a)$ implies $a = b$, and
$(a \preceq b)$ and $(b \preceq c)$ implies $a \preceq c$.)  

\begin{definition}
An update function is {\em monotonic} for $(\preceq,\Omega)$ if 
\[
(\forall x, y \in \Omega)(\forall u \in [0,1])
  (x \preceq y \Rightarrow \phi(x,u) \preceq \phi(y,u)).
\]

\end{definition}

Let $\pi$ be a distribution over the state space of the Markov chain.  
Then say $\pi$ is stationary if $X_{t} \sim \pi$ implies $X_{t+1} \sim \pi$
as well.
For example, for the Markov chain defined by
$X_{t+1} = U_{t+1}^{1/\beta}(1 + X_t)$ where $U_1,U_2,\ldots$ are iid 
$\unifdist([0,1])$, the stationary distribution of the chain is
the Vervaat perpetuity.
Markov chain Monte Carlo 
takes advantage of the fact that under limited
assumptions, for any fixed $x_0$, the distribution of $X_t$ approaches
$\pi$ as $t$ goes to infinity.  

When applicable, cftp allows the user to directly
simulate from $\pi$, the stationary distribution of the Markov chain.
The intuition is to view the Markov chain 
%as given
%by monotonic update function $\phi$ 
as running for times starting from the far past, so for times 
$\{\ldots,-3,-2,-1,0\}.$ In other words, it has already been running for
an infinite amount of time up to time $t = 0$.  Having run for
an infinite number of steps, the idea is that 
$X_0$ comes exactly from the stationary distribution of the Markov chain.

More precisely, fix a time $t < 0$.  Suppose
that $M$ is the largest state in the state space, and $m$ the smallest, so
that $m \preceq x \preceq M$ for all $x \in \Omega$.  
Let $m_t = m$ and $M_t = M$.  For $r$ from $t+1$ up to 0, 
let $M_{r} = \phi(M_{r-1},U_r)$
and $m_r = \phi(m_{r-1},U_r)$.    

Since $m_t \preceq X_t \preceq M_t$, a 
simple induction gives that $m_r \preceq X_{r} \preceq M_r$ for all $r$ up to 0.
In particular, if $m_0 = M_0$ then $X_0$ also equals that common value, and 
the algorithm terminates with $X_0 \sim \pi$. 
%seems slightly more clear to use X_0 instead of m_0 here ^

If $m_0 \neq M_0$, then recursively call the algorithm to obtain 
$X_t$.  Then use $U_{t+1},\ldots,U_0$ to update $X_t$ forward to $X_0$, and output $X_0$.  Either way, the algorithm will output $X_0 \sim \pi$. % removed "again" since technically there is only ever one output, based on the way the recursion works

\begin{comment}
More precisely, suppose
that $U_t$ is the random choice used to update state $X_{t-1}$
to $X_t$.  That is, $X_t = \phi(X_{t-1},U_t)$ for all $t$.  Further suppose
that $M$ is the largest state in the state space, and $m$ the smallest, so
that $m \preceq x \preceq M$ for all $x \in \Omega$.  Then for all $t \leq 0$,
let $m_t^t = m$ and $M_t^t = M$.  For $r > t$, let $M^t_{r} = \phi(M^t_{r-1},U_r)$
and $m^t_r = \phi(m^t_{r-1},U_r)$.  

Similarly, for any $x \in \Omega$, let 
$x_t^t = x$ and $x^t_r = \phi(x^t_{r-1},u_t)$.  Because $\phi$ is 
monotonic and $m_t^t \preceq x_t^t \preceq M_t^t$, a simple induction proof
gives that $m_r^t \preceq x_r^t \preceq M_r^t$ for any $r > t$.  In particular,
if $m_0^t = M_0^t$, then the value of
$X_0$ must equal this common value as well.

So monotone coupling from the past works by generating $U_{t+1},\ldots,U_{0}$
for some $t < 0$ and then using $\phi$ to compute $m_0^t$ and $M_0^t$.  If
they are the same, output this value and quit.  Otherwise, let $t_{\text{old}}$
be the old value of $t$, make $t$ smaller than it was, generate
$U_{t},U_{t+1},\ldots,U_{\told - 1},U_{\told}$
and repeat until 
$m_0^t = M_0^t$.
\end{comment}

\subsection{Dominated cftp}

The basic monotone cftp method cannot be used here because the state
space is $[0,\infty)$.  So no
state $M$ is an upper bound on $\Omega$.  Kendall and 
M{\o}ller~\cite{kendall1995,kendallm2000} solved this issue by introducing
dominating coupling from the past (dcftp), also known as 
coupling into and from the past.

In their approach, a second update function $\phi_{\text{D}}$ is needed which
dominates the original update function in the sense that
\[
(\forall x,w\in\Omega)(\forall u \in [0,1])(x \preceq w \Rightarrow
  \phi(x,u) \preceq \phi_{\text{D}}(w,u)).
\]
Call the Markov chain created by update function $\phi_D$ the 
{\em dominating chain}.  The chain must have the following properties.
\begin{itemize}
  \item
  The dominating chain must have a stationary distribution $\pi_{\text{D}}$
  \item
  It must be possible to draw from $\pi_{\text{D}}$.
  \item
  It must be possible to run the dominating chain $D_t$ backwards from
  stationarity, that is,
  to draw $D_{t - 1}$ given $D_t \sim \pi$.
  \item
  It must be possible to impute the forward $U_t$ values.  That is,
  using the relationship 
  $D_t = \phi_D(D_{t-1},U_t)$, it must be possible to simulate from
  $[U_t|D_t,D_{t-1}]$.
\end{itemize}

With such a dominating chain in place, dcftp runs as follows.  Assume
$t < 0$ is a parameter given to the program, and that $m$ is a minimum
state of the chain, so $(\forall x \in \Omega)(m \preceq x)$.
\begin{enumerate}
  \item
  Draw $D_0$ from $\pi_{D}$.  Set $t_{\text{old}}$ to 0.
%  \item
%  Draw $U_{t+1},\ldots,U_0$ iid uniform over $[0,1]$.
  \item
  For $i$ from $\told-1$ down to $t$, draw $D_i$ given $D_{i+1}$.
  \item
  For $i$ from $t+1$ to $\told$, draw $U_{i}$ to be uniform over $[0,1]$
    conditioned on $\phi_D(D_{i-1},U_i) = D_i$.  
%  \item
%  For $i$ from $\told-1$ down to $t$, draw $D_{i}$ given $D_{i+1}$ and $U_{i+1}$.
  \item
  Set $M_t$ to $D_t$, and $m_t$ to $m$.
  \item
  For $i$ from $t + 1$ to $0$, let $m_i$ be $\phi(m_{i-1},U_i)$ and 
   $M_i$ be $\phi(M_{i-1},U_i)$
  \item
  If $M_0 = m_0$, then output this common value and quit.  Else,
  set $t_{\text{old}}$ to $t$, $t$ to $2t$, and return to line 2.
\end{enumerate}

Note that by setting $t$ to $2t$ in the last line, the value of $t$ quickly
grows to the size needed to have a good chance that 
$m_0 = M_0$.  Setting $t$ to $t+1$ minimizes the number of random draws that
are generated, but if $\phi$ takes a long time to calculate, this approach
can result in a much longer computational time.

\subsection{The earlier method}
The monotone dominated cftp method of~\cite{huber2010a} operated as follows.
First, the update function used the following fact about 
uniforms raised to powers.

% made this x >= 0 because 0 is in our state space (right?)
\begin{lemma}  Let $x \geq 0$, and
\label{LEM:first}
$W = U^{1/\beta}$ where $U \sim \unifdist([0,1])$.  
Then 
\[
[(1+x)W|(1+x)W \leq 1] \sim W.
\]
\end{lemma}

\begin{proof}
Suppose $\beta = 1$.  Then $(1 + x)W \sim \unifdist([0,1+x])$.  Conditioning
on $(1 + x)W \leq 1$ is the same as conditioning on $W \leq 1/(1+x)$.  
A uniform conditioned on lying in a smaller space is uniform over that
smaller space, so $[W|(1+x)W \leq 1] \sim \unifdist([0,1/(1+x)])$ which makes 
$[(1+x)W|(1+x)W \leq 1] \sim \unifdist([0,1]) \sim W.$

Now for $\beta \neq 1$.  Note $((1+x)W)^\beta \sim U\cdot (1+x)^\beta$. If
$(1+x)W \leq 1$, then $((1+x)W)^\beta \leq 1$.  From the $\beta = 1$ case,
conditioned on $U(1+x)^\beta \leq 1$, we have 
$U\cdot (1+x)^\beta \sim \unifdist([0,1])$.  So 
\[
[U \cdot (1+x)^\beta|W(1+x) \leq 1] \sim U.
\]
Raising both sides to the $1/\beta$ power then gives the result.
\end{proof}

This gives rise to the following update function.  Draw two uniforms.  
The first uniform $U(1)$ determines if $U(1)^{1/\beta}(1+x) \leq 1$.  If so,
then the second uniform $U(2)$ 
is used to set the next state to $U(2)^{1/\beta}$.
Otherwise $U(1)^{1/\beta}(1+x) > 1$, and the next state should be
$U(1)^{1/\beta}(1+x).$

Let $\ind(\cdot)$ denote the usual indicator function that evaluates
to 1 if the Boolean argument is true, and is 0 otherwise.  
With this notation, let 
$S(x,u) = \ind(u < 1/(1+x)^\beta)$, and then the update function becomes
\begin{equation}
\phi(x,u(1),u(2)) = 
  S(x,u(1))u(2)^{1/\beta} + 
  [1-S(x,u(1))]u(1)^{1/\beta}(1+x).
\end{equation}

The key property of this update function, is that if $u(1) < 1/(1+x)^\beta$,
then the value of $\phi(x,u(1),u(2))$ no longer depends on $x$!  No matter
what $x$ is at that point, $\phi(x,u(1),u(2)) = u(2)^{1/\beta}$.  So this
couples together the process, bringing 
our bounds on $x$, which used to form an interval, to the same value.
%an interval of $x$ values to the same next state.  
The chance that this coupling occurs is simply the chance
that $u(1) < 1/(1+x)^\beta$, which is $1/(1+x)^\beta$.  

Now when $\beta$ is large, $x$ has to be small before this coupling will
occur with reasonable probability.  For $U \sim \unifdist([0,1])$,
$\mean[U^{1/\beta}]=\beta/(1+\beta)$.  Then 
equation~\eqref{EQN:mean} gives that 
the expected value of a draw from the perpetuity 
is $\beta$, % <----- doesn't this require justification? 
and 
$\prob(U^{1/\beta}(1+\beta) \leq 1) = (1+\beta)^{-\beta}.$
On average, this event takes $(1+\beta)^\beta$ steps to occur.
This is what leads
to the poor running time for large
$\beta$.

The dominating function for the method is an asymmetric simple 
random walk on the shifted integers
$\{x_0 - 1,x_0,x_0 + 1,\ldots\}$.  For $\beta \in (0,\infty)$, let
\[
x_0 = \frac{1 + (2/3)^{1/\beta}}{1 - (2/3)^{1/\beta}}
\]
and
\[
\phi_{\text{D}}(x,u(1),u(2)) = x + \ind(u(1) > 2/3) - 
  \ind(u(1) \leq 2/3,x \geq x_0)
\]
[Note that this is a slight change from the $x_0$ of~\cite{huber2010a}
that simplifies the algorithm slightly.]

Since the difference between adjacent states of this chain is $\{-1,0,1\}$,
this chain is time reversible, which means that from the stationary distribution
a simulation forward in time has the same distribution as a simulation backwards
in time.  Therefore the only question is in evaluating the forward $U_t$
conditioned on $D_t$ and $D_{t-1}$.  This is easy: conditioned
on $D_t = D_{t-1} + 1$, $U_t \sim \unifdist((2/3,1])$.  Conditioned on
$D_t \leq D_{t-1}$, $U_t \sim \unifdist([0,2/3])$.

\begin{lemma}
For all $x \in [0,y]$, $y \geq x_0 - 1$, $u(1) \in [0,1]$, and $u(2) \in [0,1]$,
\[
\phi(x,u(1),u(2)) \leq \phi_{\text{D}}(y,u(1),u(2)).
\]
\end{lemma}

\begin{proof}
Suppose $u(1) > 2/3$ so $D_{t} = D_{t-1} + 1$.  
Then for all $x$ and $u(2)$, it is always true that
$\phi(x,u(1),u(2)) \leq 1 + x$, hence the result
holds.

Next suppose that $u(1) \leq 2/3$.  By the monotonicity of $\phi$,
$\phi(x,u(1),u(2)) \leq \phi(y,u(1),u(2))$, so it suffices to show that
$\phi(y,u(1),u(2)) \leq \phi_{\text{D}}(y,u(1),u(2))$, or equivalently, that
\[
u(1)^{1/\beta}(1 + y) \leq y - 1.
\]
The value of $x_0$ was chosen so that 
this inequality is equivalent to $y \geq x_0$, so the result holds.

The last case to consider is when $y = x_0 - 1$.  Then
$\phi_{\text{D}}(y,u(1),u(2)) = y$.  Again, $x_0$ is large enough that
$(2/3)^{1/\beta}(1+y) \leq y$ when $y = x_0 - 1$, so the result holds
for this last case as well.
\end{proof}

\subsection{The new method}

The new method takes advantage of a simple fact:  dominated coupling from
the past works even if the update function for the underlying chain is 
changing from time step to time step.  The Markov chain itself is
time homogeneous: the distribution of $X_t$ given $X_{t-1} = x$ is
unchanging with $t$.  However, the update function used to move the chain
can be changing from step to step as long as each update function is still
an update function for the original chain, and the dominating chain is 
still dominating at each step.
That is, it is important to have a family of update
functions $\phi_t$ such that:
\[
(\forall t)(\forall x \in \Omega)
 (U \sim \unifdist([0,1]) \rightarrow \phi_t(x,U) \sim [X_{t+1}|X_t = x])
\]
and
\[
(\forall t)(\forall x \preceq w)(\forall u \in [0,1])
 (\phi_t(x,u) \preceq \phi_{\text{D}}(w,u))
\]

To take advantage of this flexibility, we need to generalize 
Lemma~\ref{LEM:first}.

\begin{lemma}
For any $x > 0$ and $a \geq 1$, 
\[
[W(1+x)|W(1+x) \leq a] \sim a U^{1/\beta}
\]
\end{lemma}

\begin{proof}
Let $U_1 \sim \unifdist([0,1])$
and $W = U_1^{1/\beta}$.  
Then $W(1+x) \leq a \Rightarrow U_1 \in [0,(a/(1+x))^\beta]$.
Let $U_2 \sim \unifdist([0,(a/(1+x))^\beta])$. Then 
$[W(1+x)|W(1+x)\leq a] \sim U_2^{1/\beta}(1+x)$.  For $U_3 \sim \unifdist([0,1])$,
$U_3(a/(1+x))^\beta \sim U_2$.  So 
\[
[W(1+x)|W(1+x)\leq a] \sim (U_3(1/(1+x))^{\beta})^{1/\beta}(1+x) = a U_3,
\]  
which completes the proof.
\end{proof}

Now suppose it is known that $m_t \leq x_t \leq M_t$.  Then the update
function at time $t$ depends on $m_t$.  That is, 
\begin{align*}
\phi_t(x,u(1),u(2)) &= \phi(x,u(1),u(2),m_t) \\
  &= r \cdot (1+m_t)u(2)^{1/\beta}
 + (1 - r) \cdot (1+x)u(1)^{1/\beta}
\end{align*}
where
\begin{align*}
r = r(x,u(1),m_t) &= 
  \ind\left(u(1) \leq \left(\frac{1+m_t}{1+x}\right)^\beta\right).
\end{align*} 

When $m_t = 0$, this is the same as the previous update function.
However, when $m_t > 0$ this can give a much improved chance of coupling
occurring.  Note that the dominating chain for this new method
is the same as the old.  This gives the following procedure for
simulating from Vervaat perpetuities.  Write $X \sim \geodist(p)$
if $\prob(X = i) = p(1-p)^i$ for $i \in \{0,1,\ldots\}$.  Unless 
specifically mentioned, all random variable draws are taken to 
be independent.

\begin{enumerate}
\item
Inputs are $\beta$ (the parameter of the Vervaat family), $\ell$ (the
inital number of steps to run), and an optional input $D_0$ 
(the value of the dominating
chain at time 0).
\item
Initialize by setting $x_0 \leftarrow (1+(2/3)^{1/\beta})/(1 - (2/3)^{1/\beta})$.
\item
If $D_0$ is given as an input to the algorithm, use it, otherwise,
let $D_0 \leftarrow x_0 - 1 + G$, where $G \sim \geodist(1/2)$.
\item
Generate $D_{-1},D_{-2},\ldots,D_{-\ell}$ using a reversible 
asymmetric simple random
walk with partially reflecting boundary at $x_0 - 1$.  That is,
for $t$ from $-1$ down to $-\ell$, draw $A \sim \unifdist([0,1])$,
then let 
\[
D_{t - 1} \leftarrow D_t + \ind(A > 2/3) - \ind(A \leq 2/3,D_t\geq x_0).
\]
\item
Set $m_{-\ell}$ equal to $0$ and $M_{-\ell}$ equal to $D_{-\ell}$.
For $t$ from $-\ell+1$ up to $0$, draw $U_t(1)$ as $\unifdist([0,2/3])$
if $D_{t} \leq D_{t-1}$, or as $\unifdist((2/3,1])$ if 
$D_t > D_{t-1}$.  In either case, draw $U_t(2) \sim \unifdist([0,1])$.  
\item
For $t$ from $-\ell+1$ to $0$, let 
\[
m_{t} \leftarrow \phi(m_{t-1},U_t(1),U_t(2),m_{t-1}),
M_t \leftarrow \phi(M_{t-1},U_t(1),U_t(2),m_{t-1}).
\]
\item
If $m_0 = M_0$, output this common value as the result and quit.  Otherwise,
perform the following steps.
First, 
call the algorithm recursively, with input for the dominating chain of 
$D_{-\ell}$, and time steps $2\ell$ ($\beta$ remains the same).  Call the 
outcome of the recursive call
$Y_{-\ell}$.  Next, 
set $m_{-\ell} \leftarrow 0$ and $M_{-\ell} \leftarrow D_{-\ell}$.  
Then for $t$ from $-\ell+1$ to $0$, let
$m_{t} \leftarrow \phi(m_{t-1},U_t(1),U_t(2),m_{t-1})$,
$M_{t} \leftarrow \phi(M_{t-1},U_t(1),U_t(2),m_{t-1})$, and
$Y_t \leftarrow \phi(Y_{t-1},U_t(1),U_t(2),m_{t-1})$.  
Finally, output $Y_0$ and quit.
\end{enumerate}

This procedure is implemented in \rsoft{} in the appendix.

The new algorithm differs in two key ways from dominated cftp of
the previous section.  First, it maintains both an upper and a lower
bound on the process, thereby speeding coalescence.  Second, since
the algorithm is being called recursively, the values of the 
lower and upper process are reset when the recursive call ends.

So for example, suppose initially $t = 10$.  Then the process is
run over 10 steps, from time $t = -10$ up to $t = 0$.  In the recursive
call, twice as many steps are used, so in this example 20 steps.  This
can be viewed as generating the value of the process from time
$t = -30$ up to $t = -10$.  If coalescence occurs at or before time
$t = -10$, when the process reaches time $-10$, 
$m_{-10}$ is still reset to 0, and $M_{-10}$ is reset to 
the value of the dominating process before the state is run forward
up to time 0.  This ensures that every step is updated according
to the same update function, using the same random choices.

\begin{figure}[ht]
\begin{center}
\includegraphics[height=3in]{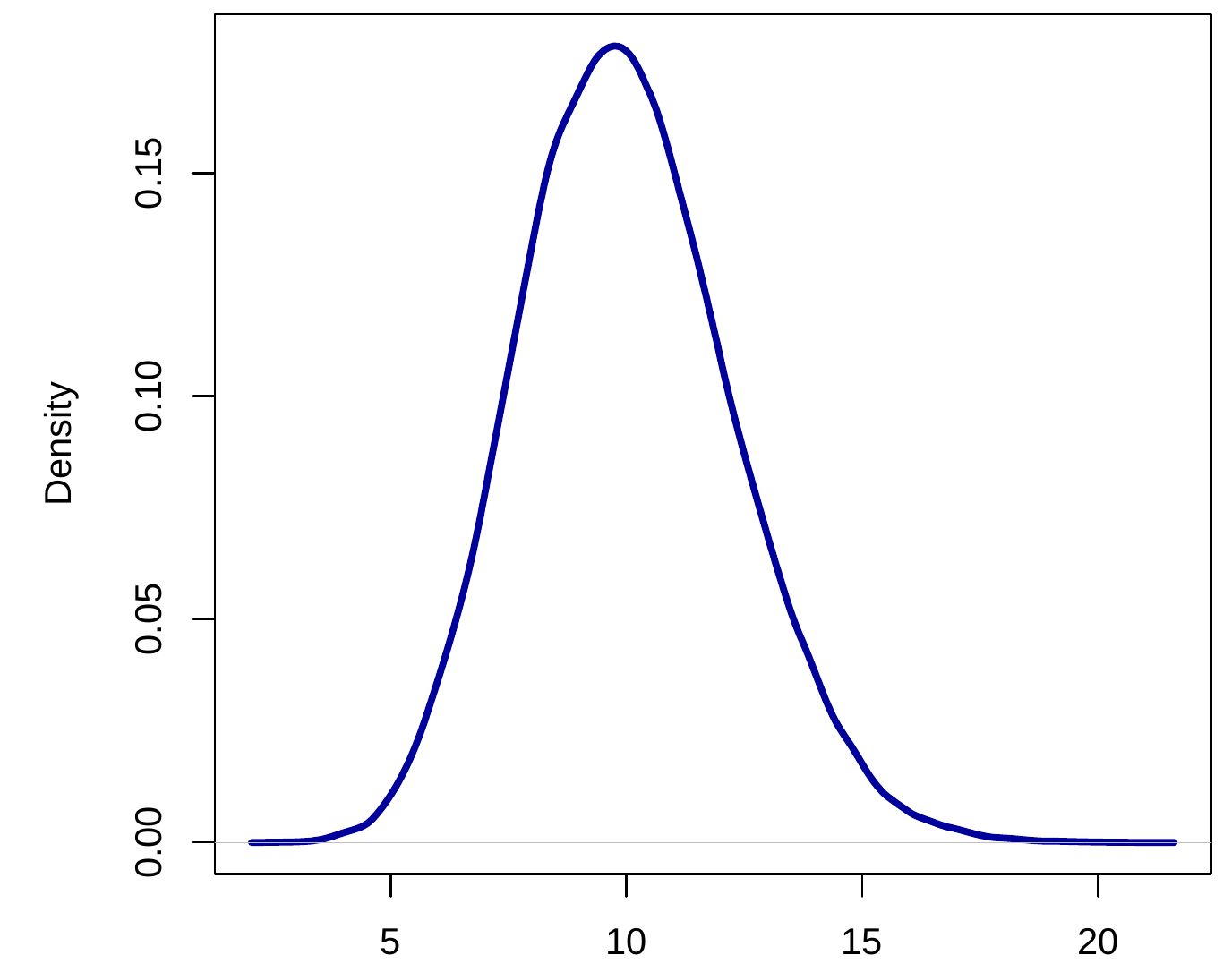}
\caption{Density of Vervaat perpetuity with $\beta = 10$ estimated
from 10,000 samples.  This $\beta$ would have required something like $(22.3)^{10}$ steps for a single sample using the old method.  
With
the new method, on the order of $10\ln(10) \approx 23$ steps per sample were required, and 10,000 samples were generated in 144 seconds using R on a Windows 10 
machine with an
Intel Core i7-6600U CPU at 2.60GHz.
}
\end{center}
\end{figure}

\section{The run time}
\label{SEC:runtime}

Now consider the average number of steps taken by the algorithm.  
This is proportional to the largest value of $\ell$ input to the algorithm,
and so to understand the running time it is necessary to under how large
$\ell$ grows on average.

\begin{lemma}
For $D_0 = x_0 - 1 + G$, and 
$\ell \geq \ln(4\delta^{-2}\beta(x_0+1))(1+\beta) + 1$
 for $\delta > 0$,  
the chance that the algorithm returns $Y$ at the first step is at least
$1 - \delta$.
\end{lemma}

\begin{proof}
  The expected value of $D_0$ is $x_0 + 1$.  
  Let $m \leq M$ and consider $m'=\phi(m,U_t(1),U_t(2),m)$,
  $M' = \phi(M,U_t(1),U_t(2),m)$.  Then
  \[
  \mean[M' - m'|m,M] = \mean[U^{1/\beta}(1+M)|M] - \mean[U^{1/\beta}(1+m)|m]
   = \frac{\beta}{\beta+1}(M - m).
  \]
  
  At the beginning of line 6, $M - m = D_{-\ell} - 0$.
  Now let $M_{-1}$ and $m_{-1}$ be the values of $M$ and $m$ 
  at the end of $\ell-1$ steps in line $6$.  Then a simple induction gives
  \[
  \mean[M_{-1} - m_{-1}|D_{-\ell}] = \left(\frac{\beta}{\beta + 1}\right)^{\ell - 1}
    D_{-\ell}.
  \]
  Taking the expected value of both sides and 
  $\beta/(\beta+1) \leq \exp(-1/(\beta + 1))$ gives
  \[
  \mean[M_{-1} - m_{-1}] \leq 
    (x_0  + 1)\exp(-(\ell - 1)/(\beta + 1)).
  \]
  Then  
  \[
  (\forall a > 0)(\prob(M_{-1} - m_{-1} \geq a) \leq 
    a^{-1}(x_0+1)\exp[-(\ell-1)/(\beta+1)])
  \]
  by Markov's inequality.

  When $U_0(1) \leq ((1+m)/(1+M))^\beta$, then $m_0 = M_0$ 
  at the final step, so the next goal is to show that 
  $((1+m)/(1+M))^{\beta}$ is close to 1.  Suppose 
  $((1+m)/(1+M)) \geq 1 - (1/2)\delta/\beta.$  It is 
  easy to show that for $\beta \geq 1$,
  \[
  \left(1 - \frac{\delta}{2\beta}\right)^\beta \geq 1 - 
    \frac{\delta}{2}.
  \]

  Since $(1+m)/(1+M) = 1 - (M - m)/(1+m)$ and $m \geq 0$,
  this means 
  \begin{equation}
  \label{EQN:bound}
  \prob(m_0 = M_0|M - m \leq (1/2)\delta/\beta) \geq 
    1 - \delta/2.
  \end{equation}

  Suppose 
  $\ell \geq \ln(4 \delta^{-2}\beta(x_0+1))(\beta + 1) + 1$.  
  Then the chance that 
  $M - m > (1/2)\delta/\beta$ is at most 
  $2\beta\delta^{-1}(x_0+1)\exp(-\ln(4\delta^{-2}\beta(x_0+1))) = \delta/2$, which together
  with~\eqref{EQN:bound} completes the proof.
\end{proof}

For $\beta \geq 1$, it holds that $\ln(x_0+1) \leq \ln(6\beta)$.
Suppose that 
\[
\ell = 
 (\beta + 1)[2\ln(\delta^{-1}) + \ln(4) + \ln(\beta) + \ln(6\beta)] + 1.
\]
Then it holds that 
\[
2\ell \geq (\beta + 1)[2\ln(\delta^{-2}) + \ln(24) + 2\ln(\beta)]+1.
\]
That is to say, if $\ell$ gives a chance of failure to recurse of 
$\delta$, then $2\ell$ steps gives a chance of failure to recurse of 
at most $\delta^2$.
Hence if the initial $\ell$ satisfies the inequality with $\delta = 1/5$,
then $2\ell$ satisfies the inequality with $\delta = (1/5)^2 = 1/25$, and
so forth.

\begin{lemma}
Let $T$ be the sum of all the values of $\ell$ in the inputs to all
the calls to the algorithm.  Then 
\[
\mean[T] \leq (5/3)((\beta+1)[2\ln(\beta)+\ln(600)] + 1).
\]
\end{lemma}

\begin{proof}
Let $\ell_i = 2^{i}$, where $i$ is the 
depth of the recursion.  Then if $R$ is the greatest level of recursion
called (and recursion level 0 refers to the initial call to the algorithm),
then 
\[
\mean[T] = \sum_{i=0}^\infty 2^i \cdot \prob(i \leq R).
\]

Let $n = (\beta + 1)[2\ln(5) + \ln(24) + 2 \ln(\beta)] + 1.$  Then
for $i < \log_2(n)$, $\prob(i \leq R) \leq 1$.  The sum of $2^i$ for these
terms is at most $n$.
 
From the previous discussion
$\prob(\lceil \log_2(n) \rceil \leq R) = 1/5$,
$\prob(\lceil \log_2(n) \rceil + 1 \leq R) \leq 1/25$, and so on.  The
sum of these $2^i \prob(i \leq R)$ terms is 
$2n/5+4n/25+\cdots = (2/3)n$.  Therefore $\mean[T] \leq (5/3)n$, which
completes the proof of the bound of the mean.
\end{proof}

Theorem~\ref{THM:main} follows immediately.
The variance can be bounded in a similar fashion.

\begin{lemma}
For $T$ as before,
\[
\mean[T^2] \leq (38/3)((\beta+1)[2\ln(\beta)+\ln(600)] + 1)^2.
\]
\end{lemma}

\begin{proof}
Start with $n$ and $R$ as in the previous proof, then 
\[
T^2 = \sum_{i=0}^\infty \sum_{j=0}^i 2\cdot 2^i 2^j \ind(i \leq R)\ind(j \leq R)
 \leq \sum_{i=0}^\infty 2^{2i+1} \ind(i \leq R).
\]
Now the sum of the expectations of the individual terms is 
$2(4/3)n^2$ for $i < \log_2(n)$, and 
$2[(4/5)n^2 + (16/25)n^2 + \cdots] = 10n^2$ for
the rest of the terms, giving the result.
\end{proof}

\section{Acknowledgments}

This work was supported by NSF DMS-1418495.

%\section{Non-Vervaat perpetuities}
%\label{SEC:generalize}

%Certain perpetutities where 

\appendix

\section{Code}

The following code implements the Vervaat perpetuity algorithm of
Section~\ref{SEC:method} for \rsoft{}
version 3.0.2 (2013-09-25).

\begin{verbatim}
vervaat <- function(beta = 1,steps = 1,d = -1) {
# Written by Kirkwood Cloud and Mark Huber 20 August, 2015

  # Line 2
  x0 <- (1+(2/3)^(1/beta))/(1-(2/3)^(1/beta))
  # Line 3
  if (d == -1) d <- x0 - 1 + rgeom(1,prob=1/2)
  # Line 4
  d <- c(rep(0,steps),d); a <- runif(steps)
  for (t in steps:1) d[t] <- d[t+1]+(a[t]>2/3)-(a[t]<=2/3)*(d[t+1]>=x0)
  # Line 5
  m <- 0;M <- d[1]; u1 <- rep(0,steps); u2 <- runif(steps)
  for (t in 2:(steps+1)) {
    up <- d[t]>d[t-1]; u1[t - 1] <- runif(1,min=2/3*up,max=2/3+1/3*up)
  }
  # Line 6
  for (t in 2:(steps+1)) {
    m <- (1+m)*u2[t-1]^(1/beta) 
    a <- u1[t-1]^(1/beta); s <- (a < ((1+m)/(1+M))); M <- s*m+(1-s)*a*(1+M)
  }
  # Line 7
  if (m == M) return(m) else {
    y <- vervaat(beta,2*steps,d[1]); m <- 0; M <- d[1]
    for (t in 2:(steps+1)) {
      r <- (u1[t-1] < ((1+m)/(1+y))^beta) 
      m <- (1+m)*u2[t-1]^(1/beta); y <- r*m+(1-r)*u1[t-1]^(1/beta)*(1+y)
    }
    return(y)
  }
}
\end{verbatim}

\bibliographystyle{plain}

\end{document}